\documentclass[11pt, baustms]{amsproc}
\usepackage{amssymb, amsmath, enumerate,booktabs, tikz, pgfplots,color}
\usepackage[skins,breakable]{tcolorbox}
\usepackage{graphicx,subcaption}
\usepackage{listings}
\usepackage[absolute,overlay]{textpos}
\usepackage{orcidlink}
\usepackage{url,hyperref,parskip}
\usepackage{thmtools}

\usetikzlibrary{math}
\usetikzlibrary{arrows}
\usetikzlibrary{patterns}
\usepackage{pgfplots}
\pgfplotsset{compat=1.18}
\usepackage{tikz}
\usepackage{atbegshi}
\hypersetup{citecolor=blue, linkcolor=blue, colorlinks=true}
\usepackage[marginparwidth=1.5cm,margin=2cm]{geometry}
\usepackage{cleveref}
\newtheoremstyle{scplain}
  {.5\baselineskip plus .2\baselineskip minus .2\baselineskip}
  {.5\baselineskip plus .2\baselineskip minus .2\baselineskip}
  {\itshape}
  {\parindent}
  {\scshape}
  {.}
  {5pt plus 1pt minus 1pt}
  {}

\theoremstyle{scplain}
\newtheorem{theorem}{Theorem}[section]
\newtheorem{lemma}[theorem]{Lemma}

\newtheorem{corollary}[theorem]{Corollary}

\newtheorem*{rexample}{Example}
\tcolorboxenvironment{rexample}{
    blanker,
    breakable,
    left=5mm,
    before skip=10pt,
    after skip=10pt,
    borderline west={5pt}{0pt}{gray!65},
    pad before break=2pt,
    pad after break=2pt,
    use color stack,
}

\newcommand\CA{\mathrm{CA}}
\newcommand\CE{\mathrm{CE}}

\renewcommand\le{\leqslant}
\renewcommand\ge{\geqslant}

\title[On applications of the clique-adjacency polynomial]{On applications of the clique-adjacency polynomial\\ to arbitrary finite graphs}

\author{John Bamberg \orcidlink{0000-0001-7347-8687}}
\address[Bamberg]{Department of Mathematics and Statistics, 
The University of Western Australia, 35 Stirling Highway, Perth, W. A. 6009, Australia.}

\author{Jake Rigg \orcidlink{0009-0006-4786-027X}}
\address[Rigg]{Department of Mathematics and Statistics, 
The University of Western Australia, 35 Stirling Highway, Perth, W. A. 6009, Australia.}

\date{}



\begin{document}

\begin{abstract}
The clique adjacency polynomial (CAP), introduced by Soicher (2015), provides a powerful method for bounding the clique numbers of edge-regular graphs. In this paper, we extend the CAP framework to arbitrary finite graphs by expressing the relevant parameters in terms of average vertex degree and average edge-degree over potential cliques. This leads to a generalised CAP bound and an associated clique existence polynomial (CEP), which removes the dependence on an auxiliary integer variable and facilitates computation.
We compare the resulting bounds with classical spectral and linear programming bounds, including those of Delsarte, Hoffman, and Haemers. We show that the generalised CAP improves upon these bounds for several families of graphs. In particular, we identify infinite families of edge-regular graphs arising from projective geometry for which the CAP outperforms the Delsarte bound, as well as families of regular and non-regular graphs where the generalised CAP improves upon the Hoffman and Haemers bounds. We also develop techniques for bounding feasible parameter regions, enabling practical application of the method to both structured and unstructured graphs.
\end{abstract}

\vspace*{-1cm}
\maketitle

\section{Introduction}

Improving methods for bounding the clique number of a graph reduces the need for extensive computational search and enables deeper analytical insight into clique structure. In his 1973 thesis, Philippe Delsarte introduced a novel linear programming approach for bounding cliques in graphs arising from association schemes \cite{delsarte}. This result, known as the \emph{Delsarte linear programming bound}, simplifies to the \emph{Delsarte bound} when applied to a single graph within an association scheme. Shortly thereafter, Alan Hoffman extended the Delsarte bound to all regular graphs, and Willem Haemers subsequently generalised it further to arbitrary graphs.\footnote{These generalisations are referred to as the Hoffman bound and the Haemers bound, respectively.} Together, these bounds have become fundamental tools for estimating clique numbers and are widely used across graph theory and its applications.

In 2010, Leonard Soicher introduced a new bound designed for the study of edge-regular graphs \cite{soicher_more_2010}, which he termed the \emph{Clique Adjacency Polynomial (CAP)}. This bound has since played a significant role in the analysis of cliques in both edge-regular and strongly regular graphs, with further developments by Soicher \cite{soicher_edge_2015}, and Greaves \& Soicher \cite{grv}. Greaves and Soicher showed that the CAP improves upon the Delsarte bound for ``infinitely many feasible parameter tuples for strongly regular graphs" \cite{grv}.
In this paper, we investigate how the CAP may be extended beyond edge-regular graphs to provide bounds on the clique numbers of general graphs.
For a clique $S$ of a possibly non-regular and non-edge-regular graph $\Gamma$, define averages $\langle k\rangle_S:=\frac1{|S|}\sum_{u\in S}\deg_\Gamma(u)$ and $\langle\lambda\rangle_S:=\binom{|S|}{2}^{-1} \sum_{\{u,w\}\in\binom{S}{2}}|\Gamma(u)\cap\Gamma(w)|$.

\begin{theorem}
    \label{gen}   
    Let $\Gamma$ be a graph on $v$ vertices and let $s$ be a positive integer.
    Let \(\mathcal F_s(\Gamma)\) denote a set of admissible pairs \((k,\lambda)\) containing every pair
$\left(\langle k\rangle_S,\langle \lambda\rangle_S\right)$ arising from an \(s\)-clique \(S\) of \(\Gamma\). 
Let
    \begin{equation*}
    \CA_\Gamma(x,s,k,\lambda):=x(x+1)(v-s)-2xs(k-s+1)+s(s-1)(\lambda-s+2).\label{CAP3}
    \end{equation*}
If, for every \((k,\lambda)\in \mathcal F_s(\Gamma)\), there exists an integer \(x\) such that
\[
\CA_\Gamma(x,s,k,\lambda)<0,
\]
then \(\Gamma\) contains no clique of size \(s\).
We refer to this bound as the \textbf{generalised CAP bound}.
        
\end{theorem}


After outlining the effective application of the generalised CAP bound, we present several illustrative examples in which the bound from  \Cref{gen} improves upon both the Delsarte linear programming bound and the Hoffman bound. In addition, we prove (\Cref{edge_reg_1} and \Cref{edge_reg_2}) that the CAP bound surpasses the Delsarte bound for two infinite families of graphs, $\Gamma_{q^+}$ and $\Gamma_{q^-}$, arising from conics in Desarguesian projective planes. Finally, we give two infinite families of strictly edge-regular graphs for which the CAP improves upon the Delsarte bound. We also exhibit an infinite family of regular graphs where the generalised CAP improves on the Hoffman bound, as well as an infinite family of graphs for which it improves on the Haemers bound (\Cref{reg_1} and \Cref{nonreg_1}).

\section{Background}

  \subsection{Delsarte's clique bounds}

    An association scheme $\mathcal{A} = (V, \mathcal{R})$ can be represented as a colouring of $K_{|V|}$, where $i$-coloured edges correspond to vertices related by $R_i \in \mathcal{R}$. For $X \subseteq \mathcal{R}$, we define the regular graph $\Gamma_X$ as the graph on $V$ arising from the union of $X$. In his 1973 thesis \cite{delsarte}, Delsarte provides a novel way to bound the clique number of an association scheme graph $\Gamma_X$. Given an $n$-class association scheme $\mathcal{A} =(V, \mathcal{R})$, consider a non-empty subset $S \subseteq V$. Delsarte defines the \textit{inner distribution vector}  $a := (a_0,a_1,\dots, a_n)$ of $S$ where: 
    \begin{equation*} a_i = \frac{|(S \times S) \cap R_i|}{|S|},
	\end{equation*}
	 for $R_i \in \mathcal{R}$. By considering the inner distribution vector with respect to a maximal choice of $X \subseteq \mathcal{R}$, \footnote{By convention, $a_0=1$ always, regardless of whether $R_0 \in X$.} Delsarte produces the following linear program.
    \begin{align*}
	\textsc{maximise: }  & \sum_{i=0}^{n} a_i, \tag{DELSARTE}\\
	\textsc{subject to: } & (a Q)_j \ge 0 \textnormal{ for all entries } j, \\
	& a_0 = 1, \\
	& a_i \ge 0 \textnormal{ for } R_i \in X,\\
	& a_i = 0 \textnormal{ for } R_i \notin X. 
	\end{align*}
	Here, $Q$ is the second eigenmatrix of the relevant association scheme. The objective value of the linear program DELSARTE provides an upper bound on $\omega(\Gamma_X)$, which is referred to as the \textit{Delsarte LP bound}. Delsarte also observed the following:
    
	\begin{theorem}[{\cite[3.23]{delsarte}}]
	\label{DB}
	For a single graph $\Gamma_{R_i}$ of an association scheme, the Delsarte LP bound (a feasible solution to \textnormal{DELSARTE}) simplifies to:
	\begin{equation*} 
		\omega(\Gamma) \le 1 - \frac{k_i}{\gamma_{v}}, 
	\end{equation*}
	where $k_i$ and $\gamma_v$ are the valency and minimum eigenvalue of $\Gamma_{R_i}$ respectively.
	\end{theorem}

   The displayed inequality in Theorem \ref{DB} is referred to as the \emph{Delsarte bound}. 

	\subsection{The Hoffman bound}

    When considering a regular graph $\Gamma$ which does not arise from an association scheme, the Delsarte bound does not necessarily hold. For this circumstance we could instead apply the \emph{Hoffman bound} -- a generalisation of the Delsarte bound to arbitrary regular graphs. Hoffman never formally published his bound, however an outline of its context can be found in Haemers' 2021 paper \cite{haemers_hoffmans_2021}. Traditionally, the Hoffman bound is given in terms of the independence number, however we consider it in the following form.
    
\begin{theorem}[{\cite[Theorem 1]{haemers_hoffmans_2021}}]
	\label{HBB}	
	For a $k$-regular graph $\Gamma$ on $v$ vertices:
	\begin{equation*}
		\omega(\Gamma) \le \frac{v(1+\gamma_2)}{v-k+\gamma_2}, 
	\end{equation*}
    where $\gamma_2$ is the second-largest eigenvalue of $\Gamma$.
\end{theorem}

   \subsection{The Haemers bound}
   
Similarly, when considering a non-regular graph $\Gamma$, the Hoffman bound does not necessarily hold. The Haemers bound is a generalisation of the Hoffman bound to all graphs.

\begin{theorem}[{\cite[2.1.3]{thesis}}]\label{HaB}
	For a non-complete graph $\Gamma$ on $v$ vertices: 
	\begin{equation*} 
		\omega(\Gamma) \le \frac{-v \gamma'_1\gamma'_{v}}{(v-\Delta_\Gamma-1)^2-\gamma'_1\gamma'_{v}},
	\end{equation*}
	where $\gamma'_1, \gamma'_{v}$ are the maximum and minimum eigenvalues of $\overline{\Gamma}$ respectively and $\Delta_\Gamma$ corresponds to the maximum degree of $\Gamma$.
\end{theorem}

\subsection{The Clique Adjacency Polynomial}

Soicher first realised the CAP as a specific case of the \textit{block intersection polynomial} (BIP) \cite{soicher_more_2010}. We will introduce the CAP in the same way as Soicher does, introducing relevant concepts in the derivation of the BIP -- a polynomial intended for analysis of block-design structures. For a graph $\Gamma$, and $S, T, W \subseteq V(\Gamma)$ we let 
\begin{equation*} 
		n_i (\Gamma, S, W ) := \left|\{u \in W : \left| \Gamma(u) \cap S\right|=i\}\right|,
\end{equation*} 
and call it the \emph{$i$-th intersection number}. The quantity $n_i (\Gamma, S, W )$ counts the number of vertices in $W$ such that $\left|\Gamma(u) \cap S\right| = i$. We also define
\begin{equation*}
    \lambda_T (\Gamma, W) := |\{u \in W: T \subseteq \Gamma(u)\}|. 
\end{equation*}
Here, $\lambda_T(\Gamma, W)$ counts the number of vertices in $W$ adjacent to all vertices in $T$. 
Let $s := |S|$. For $0 \le j \le s$, define
\begin{equation*} 
    \lambda_j (\Gamma, S, W) := \binom{s}{j}^{-1} \sum_{T \subseteq S, |T|=j}\lambda_T (\Gamma, W).
\end{equation*}  
The quantity $\lambda_j (\Gamma, S, W)$ is the average $\lambda_T (\Gamma, W)$ across all $j$-sized subsets $T$ of $S$. 

Consider a strongly regular graph $\Gamma$ with parameters $(v,k,\lambda, \mu)$. For $j=2$, the value $\lambda_T(\Gamma, W)$ is not necessarily constant for a given $S$ because there are two potential $j$-sized subsets to consider -- adjacent and non-adjacent vertices \footnote{That is, $\lambda_T(\Gamma, W) = \lambda-(s-2)$ or $\lambda_T(\Gamma, W) = \mu-(s-2)$}. However, notice if $S$ corresponds to a clique of $\Gamma$, then the value $\lambda_T(\Gamma, W)$ for $j=2$ is indeed constant because there are no non-adjacent vertices in $S$. In particular, for an $s$-clique we have that $\lambda_T(\Gamma, W) = \lambda-(s-2)$ for all $2$-subsets. We provide all the $\lambda_j(\Gamma, S, W)$ values for $S$ an $s$-clique and $W = V(\Gamma) \setminus S$ in Table \ref{tab:gen1}.

	\begin{table}[!ht]
		\centering
		\begin{tabular}{c c c c}
			$j$ & $T$ & $\lambda_T (\Gamma, W)$ & $\lambda_j (\Gamma, S, W)$ \\
			\midrule
			0 & $\varnothing $ & $v-s$ & $v-s$ \\
			1 & $u \in V(S)$ & $k-s+1$ & $ k-s+1$ \\
			2 & $\{u,v\}$ for $u,w \in V(S)$ & $\lambda-s+2$ & $\lambda-s+2$ \\
			\bottomrule
		\end{tabular}
		\caption{Different $\lambda_j$ values for a strongly regular graph when $S$ corresponds to an $s$-clique and $W = V(\Gamma) \setminus S$.}
		\label{tab:gen1}
	\end{table}

Since there is no non-edge valency requirement $\mu$ in Table \ref{tab:gen1}, these values also hold for edge-regular graphs. 
We can also link $n_i (\Gamma, S, W )$ to $\lambda_j (\Gamma, S, W)$ through a counting argument explored by Soicher.

    \begin{theorem}[Soicher, {\cite[Theorem 2.1]{soicher_more_2010}}]
    \label{counts}
    Let $\Gamma$ be a graph, let $S, W \subseteq V(\Gamma)$, with $s:=|S|$. For $0 \le j \le s$ we have
    \begin{equation*}
        \sum^s_{i=0} \binom{i}{j} n_i (\Gamma, S, W ) = \binom{s}{j}\lambda_j (\Gamma, S, W)
    \end{equation*}
    \end{theorem}
    
Now, we introduce Soicher's BIP. For $z$ a non-negative integer, define the polynomial 
\[
P(x,z) := x(x -1) \cdots (x-z + 1),
\]
and for real number sequences $[m_0, \dots ,m_s]$, $[\lambda_0, \dots , \lambda_t ]$, with $t \le s$, Soicher defines the block intersection polynomial
	\begin{equation}
        \label{BIP1}
	    B(x, [m_0, \dots ,m_s], [\lambda_0, \dots , \lambda_t ]) :=
\sum^t_{j=0}\binom{t}{j} P(-x, t - j) \left[P(s, j)\lambda_j - \sum^s_{i=j}P(i, j)m_i\right].
	\end{equation}
	
	Using the equation from Theorem \ref{counts} and (\ref{BIP1}) we can rearrange the BIP into a useful form. 

    \begin{corollary}[Soicher, {\cite[Theorem 3.1]{soicher_more_2010}}]
    \label{rearranged}
        Let $s$ and $t$ be non-negative integers, with $s \ge t$, let $n_0, \dots ,n_s$, $m_0, \dots ,m_s$, and $\lambda_0, \dots , \lambda_t$ be
real numbers, such that 
\begin{equation*}
        \sum^s_{i=0} \binom{i}{j} n_i = \binom{s}{j}\lambda_j 
\end{equation*} 
      Then  \begin{equation*} B(x, [m_0, \dots ,m_s], [\lambda_0, \dots , \lambda_t ]) = \sum_{i=0}^s P(i-x,t)(n_i -m_i)\end{equation*}
    \end{corollary}

    Furthermore for a BIP of this form, Soicher provides the following:

    \begin{theorem}[Soicher, {\cite[Theorem 3.1]{soicher_more_2010}}]
    \label{reform}
        Consider equation \begin{equation*} B(x) = \sum_{i=0}^s P(i-x,t)(n_i -m_i)\end{equation*} with real number sequences $[m_0, \dots ,m_s], [n_0, \dots ,n_s]$. If $m_i \le n_i$ and $t$ is even, then for every integer $m$, we have $B(m) \ge 0$.
    \end{theorem}
    
The following is an immediate consequence of Soicher's block-intersection polynomial theorem applied with \(t=2\) and \(m_i=0\) for all \(i\).

	\begin{theorem}[Soicher, {\cite[Theorem 3.1]{soicher_more_2010}}]
		\label{BIP}
	Suppose a graph $\Gamma$ contains a clique of size $s$, with vertex set denoted by $S$. Define $W := V(\Gamma) \setminus S$ and $\lambda_j:=\lambda_j(\Gamma, S, W)$. Then $B(x, [0^{s+1}], [\lambda_0, \lambda_1, \lambda_2]) \ge 0$ for every integer $x$.
	\end{theorem}
  
    We can use the contrapositive of \Cref{BIP} to consider when cliques of a certain size cannot exist. 
    For a graph $\Gamma$, if we can represent $[\lambda_0, \lambda_1, \lambda_2]$ in terms of the size of a potential clique $s$, we can simply check for the first $s$ when $B(x, [0^{s+1}], [\lambda_0, \lambda_1, \lambda_2]) < 0$ and produce an upper bound on $\omega(\Gamma)$.

    Soicher naturally applied \Cref{BIP} to the existence of $s$-cliques in edge-regular graphs for which, as we have seen in Table \ref{tab:gen1}, $\lambda_j$ values are constants in terms of $s$ and graph parameters $(v,k,\lambda)$. He called this polynomial the CAP, and initially presents it in the form
	\begin{equation}
		\CA_\Gamma(x,s) := B(x, [0^{s+1}], [\lambda_0,\lambda_1,\lambda_2]), \label{CAP}
	\end{equation}
	where $\Gamma$ is a $(v,k,\lambda)$ edge-regular graph, $S \subseteq V(\Gamma)$ which resembles a potential clique of size $s\ge 2$, $W = V(\Gamma) \setminus S$ and $\lambda_j := \lambda_j(\Gamma,S,W)$. Expanding this by using (\ref{BIP1}), we see that
	\begin{equation}
		\CA_\Gamma(x,s) = x(x+1)\lambda_0 -2x s \lambda_1+s(s-1)\lambda_2.  \label{CAP2}
	\end{equation}

	Using Table \ref{tab:gen1} we can fill in the values of $\lambda_j$ to obtain the final form of Soicher's CAP, which is how we, and Soicher in his later papers \cite{soicher_edge_2015}, will define it. Let $\Gamma$ be an edge-regular graph, with parameters $(v,k, \lambda)$. The CAP is presented as:
	\begin{equation*}
    \label{CAP1}
		\CA_\Gamma(x,s) := x(x+1)(v-s) -2x s (k-s+1)+s(s-1)(\lambda -s+2). 
	\end{equation*}
    Thus, if this quadratic is negative for some integer $x$, then no $s$-clique exists.

\begin{rexample}
    Consider $\Gamma$ as the Paley graph on $17$ vertices. This is a $(17,8,3)$ edge-regular graph, and provides an example of how the CAP can be used effectively. It is well known that $\omega(\Gamma) = 3$.
From the graph of $\CA_\Gamma(x,4)$ (see \Cref{paleyex}), we see that $\CA_\Gamma(1,4) < 0$, so a clique of size $4$ cannot exist.
\end{rexample}
\begin{figure}[!ht]
\begin{tikzpicture}
\begin{axis}[
    axis lines=middle,
    xlabel={$x$},
    ylabel={$\CA_\Gamma(x,4)$},
    domain=-1:5,
    samples=100,
    grid=both,
    width=10cm,
    height=6cm,
    enlargelimits=true,
    ymin=-10, ymax=50,
    thick
]
    \addplot[blue, thick] {13*x^2 - 27*x + 12};

\end{axis}
\end{tikzpicture}
\caption{$\CA_\Gamma(x,4)$ for $(17,8,3)$ edge-regular graph $\Gamma$}
\label{paleyex}
\end{figure}
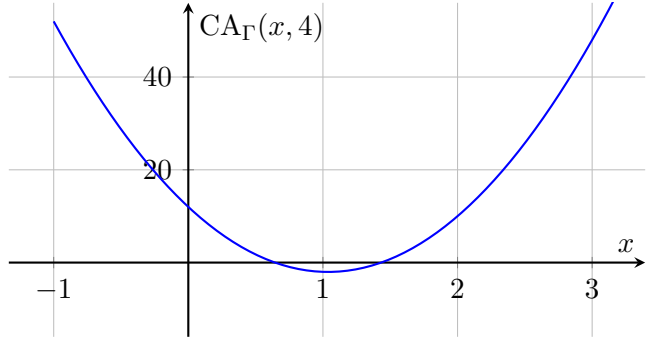

\section{Generalising Soicher's result to all graphs and further discussion}

Soicher's form of CAP hinges on \Cref{counts}, however this identity, as proven by Soicher, holds for all types of graphs, not only strictly edge-regular ones. Consider the CAP in the form of (\ref{CAP}). Instead of naturally taking $\Gamma$ to be edge-regular as Soicher did, we drop this requirement. To obtain a final formula for the generalised CAP we must now evaluate the $\lambda_j$ coefficients generally, where $S$ corresponds to a potential clique in any graph. We present this in Table \ref{tab:gen} where \(\langle k\rangle_S\) is the average degree of the vertices in \(S\), and \(\langle\lambda\rangle_S\) is the average number of common neighbours in \(\Gamma\) over unordered pairs of vertices in \(S\). These expressions hold under the assumption that $S$ induces a clique, so that each vertex in $S$ has exactly $s-1$ neighbours within $S$.

	\begin{table}[h]
		\centering
		\begin{tabular}{c c c c}
			$j$ & $T$ ($j$-sized subsets of $V(S)$) & $\lambda_T (\Gamma, W)$ & $\lambda_j (\Gamma, S, W)$ \\
			\midrule
			0 & $\varnothing $ & $v-s$ & $v-s$ \\
			1 & $u \in V(S)$ & $\deg(u)-s+1$ & $\langle k \rangle_S -s+1$ \\
			2 & $\{u,w\}$ for $u,w \in V(S)$ & $|\Gamma(u) \cap \Gamma(w)|-s+2$ & $\langle \lambda \rangle_S -s+2$ \\
			\bottomrule
		\end{tabular}
		\caption{Different $\lambda_j$ values for $S$ corresponding to a clique in $\Gamma$ where $W = V(\Gamma) \setminus S$.}
		\label{tab:gen}
	\end{table}

    \begin{proof}[Proof of \Cref{gen}]
Suppose, for contradiction, that \(S\) is an \(s\)-clique of \(\Gamma\). Put \(W=V(\Gamma)\setminus S\). By \Cref{tab:gen},
$\lambda_0=v-s$, $\lambda_1=\langle k\rangle_S-s+1$, $\lambda_2=\langle\lambda\rangle_S-s+2$. Therefore,
\[
B(x,[0^{s+1}],[\lambda_0,\lambda_1,\lambda_2])
=
\CA_\Gamma(x,s,\langle k\rangle_S,\langle\lambda\rangle_S)
\]
where
	\begin{equation}
		\CA_\Gamma(x,s) := x(x+1)(v-s) -2x s (\langle k \rangle_S -s+1)+s(s-1)(\langle \lambda \rangle_S -s+2). 
	\end{equation}
By \Cref{BIP}, this quantity is non-negative for every integer \(x\). This contradicts the assumed existence of an integer \(x\) for the feasible pair arising from \(S\). Hence no \(s\)-clique exists.    
	\end{proof}

    Note in the case of an edge-regular graph $\Gamma$, (\ref{CAP3}) is equivalent to the form of (\ref{CAP1}).
	
	For a potential $s$-clique $S$, the constants $\langle k \rangle_S$ and $\langle \lambda \rangle_S$ could take on a range of feasible values for a given $s$. It helps to consider the CAP in the form presented in (\ref{CAP1}), but with $k$, $\lambda$ as new variables representing feasible $\langle k \rangle_S$ and $\langle \lambda \rangle_S$ respectively. So we present the generalised CAP as $\CA_\Gamma(x,s,k,\lambda)$, in four input variables. 
    
    The next step is understanding how generalised $(k, \lambda)$ variables affect CAP values. To assist with further analysis we introduce a new polynomial called the Clique Existence Polynomial (CEP) denoted as $\CE_\Gamma(s,k,\lambda)$ which removes the $x$ variable from the CAP.
    
    \begin{lemma}
\label{rearrange}
	Let $\Gamma$ be a graph with parameters $(v,k,\lambda)$ for a potential clique of size $s<v$. We define the following terms.
    \begin{align*} 
    \textnormal{T}_{1_\Gamma}(s) &:= (v - (3 + 2k)s + 2s^2), \\ 
	\textnormal{T}_{2_\Gamma}(s) &:= 4(v - s)((3 + \lambda)s^2 - s^3 - (2 + \lambda)s),	\\
	\textnormal{T}_{3_\Gamma}(s) &:= \left(v - s\right)\left(1 - \left|2\left(\frac{-v + (3 + 2k)s - 2s^2}{2(v - s)}- \left\lfloor \frac{-v + (3 + 2k)s - 2s^2}{2(v - s)} \right\rfloor \right) -1 \right| \right).
    \end{align*}
If $\CE_\Gamma(s,k,\lambda) := \textnormal{T}_{1_\Gamma}(s)^2 - \textnormal{T}_{2_\Gamma}(s) - \textnormal{T}_{3_\Gamma}(s)^2 >0$, then no $s$-clique of $\Gamma$ can have average degree and edge parameters $(k,\lambda)$.
	\end{lemma}

Strictly speaking, because of the nearest-integer correction term, $\CE_\Gamma$ is a piecewise-polynomial function rather than a polynomial. We retain the name CEP to emphasise its role as the integer-existence analogue of the CAP.
This polynomial is derived by rearranging the CAP into the form $a x^2 +bx+c$. The existence of such an integer $x$ as required for the non-existence of a clique, only relies on coefficients $a,b$ and $c$ themselves. Specifically, there exists an integer $x$ such that $f(x)<0$ if, and only if, $b^2-4ac-a^2\left(1-\left|2\left(\frac{-b}{2a}- \left\lfloor \frac{-b}{2a} \right\rfloor\right)-1\right|\right)^2 > 0.$ \footnote{Note that $a = v-s \ge 0$, since a clique size is never bigger than the number of vertices in a graph. 
In the case $a=0$, the CAP reduces to a linear function, which can be checked directly.}

    Again, considering $\Gamma$ as the Paley graph on $17$ vertices we can produce another plot, this time using the CEP.
    
    	\begin{figure}[!ht]
	\begin{tikzpicture}
\begin{axis}[
    xlabel={$s$},
    ylabel={$\text{CE}_\Gamma(s)$},
	xtick={2,3,4,5,6},
    domain=2:6,
    samples at={2,3,4,5,6},
    only marks,
    mark=*,
	width=10cm,
    height=6cm,
    yticklabel style={/pgf/number format/fixed},
    extra y ticks={0},
    extra y tick labels={0},
    extra y tick style={grid=major},
    axis lines=middle,
    grid=both,
    set layers,                
    layers/axis on top/.style={axis on top}, 
    thick
]

\addplot+[
    blue,
]
{((17 - (3 + 2*8)*x + 2*x^2)^2) 
  - 4*(17 - x)*(-((2 + 3)*x) + (3 + 3)*x^2 - x^3)
  - ((17 - x)*(1 - abs(-1 + 2*((-17 + (3 + 2*8)*x - 2*x^2)/(2*(17 - x)) - floor((-17 + (3 + 2*8)*x - 2*x^2)/(2*(17 - x)))))))^2
};

\end{axis}
\end{tikzpicture}
\caption{$\CE_\Gamma(s)$ for $(17,8,3)$ edge-regular graph $\Gamma$}
\end{figure}
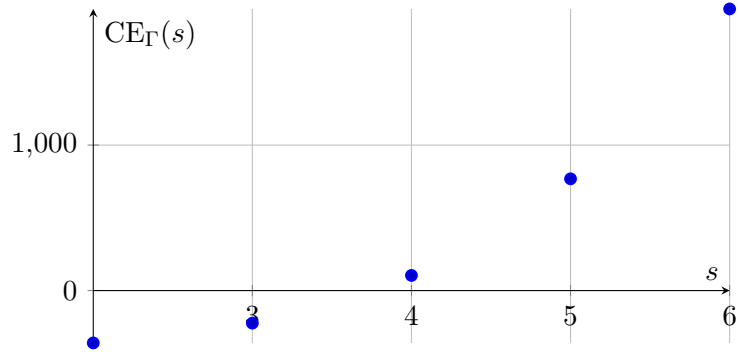

See that the same information is recovered as was in \Cref{paleyex}, a clique of size $4$, cannot exist in $\Gamma$. Now we can conveniently plot over $s$, instead of producing multiple plots over $x$ for a given $s$. The CEP is particularly useful in computation, and can be used in conjunction with the CAP to analyse the behaviour of the $k$ and $\lambda$ variables. This leads to the following lemma

    \begin{lemma}
	\label{props}
    \leavevmode
	\begin{enumerate}
    \item For fixed $s, k$, as $\lambda$ increases, $\CE_\Gamma(s,k,\lambda)$ is monotonically non-increasing.
	\item For fixed $s, \lambda$ and $k\ge s-1$, as $k$ decreases, $\CE_\Gamma(s,k,\lambda)$ is monotonically non-increasing. 
	\end{enumerate}
	\end{lemma}

While part \textit{(1)} of \Cref{props} is clear from the definition of the CAP in the form of (\ref{CAP1}), the proof of part \textit{(2)} is less intuitive. First, consider an arbitrary quadratic $f(x) = ax^2 + bx + c$, with $a > 0$ and $b \le a$. The general form of the CEP is given as $b^2-4ac-a^2\left(1-\left|2\left(\frac{-b}{2a}- \left\lfloor \frac{-b}{2a} \right\rfloor\right)-1\right|\right)^2$. Since $k$ only affects the $b$ value in the CEP, we need to consider how $b^2-a^2\left(1-\left|2\left(\frac{-b}{2a}- \left\lfloor \frac{-b}{2a} \right\rfloor\right)-1\right|\right)^2$ behaves as $b$ changes. Since $b \le a$, we can rewrite this as $b^2 - (a - |m  - a|)^2$, where $m = b \bmod 2a$. Given a fixed $a$, we can consider this as a piecewise function in terms of $b$:
		\[f(b)=
		\begin{cases}
			b^2 - m^2 & \textnormal{if } 0 \le m \le a, \\
			b^2 - (2a - m)^2 & \textnormal{if } a < m < 2a.
		\end{cases}
		\]
		\vspace{-0.1cm}
		Since $b \le a$, it means that $m = b + 2na$ for some $n \in \mathbb{N} \cup \{0\}$. We can rewrite the piecewise cases by substituting $m = b + 2na$ and simplifying:
		\[f(b)=
		\begin{cases}
			-4nab-4na^2 & \textnormal{if } 0 \le b + 2na \le a, \\
			 4(1-n)ab +4(1-n)a^2 & \textnormal{if } a < b + 2na < 2a.
		\end{cases}
		\] 
		\vspace{-0.1cm}
		\leavevmode
		Differentiating by $b$ gives: 	
		\[f'(b)=
		\begin{cases}
			-4na & \textnormal{if } 0 \le b + 2na \le a, \\
			 4(1-n)a & \textnormal{if } a < b + 2na < 2a.
		\end{cases}
		\]  
		We can see that for all relevant values of $n$ which constitute $b \le a$, we have $f(b)$ monotonically non-increasing.\footnote{Note when $n=0$, we have that $b \le a$ only holds when $f'(b) = 0$, not for when $f'(b) = 4a$, since in this case we have $b>a$.} This means that the CEP is monotonically non-increasing for $v - (3 + 2 k) s + 2 s^2 \le v-s$, which can be rearranged to $k \ge s-1$. 
        \Cref{props} leads to a useful corollary.
        
   	\begin{corollary} 
    \label{visual}
		Consider a graph $\Gamma$ and a potential $s$-clique $S$. Let $k^*_s$ be any lower bound on feasible $k$ greater than or equal to $s-1$, and $\lambda^*_s$ be any upper bound on feasible $\lambda$. If $\CE_\Gamma(s, k^*_s, \lambda^*_s) > 0$ then no clique of size $s$ can exist in $\Gamma$.
	\end{corollary} 

This is best understood visually, so we refer to \Cref{feasible region}. Recall that we associate the variable $k$ with feasible values of $\langle k \rangle_S$, and the variable $\lambda$ with feasible values of $\langle \lambda \rangle_S$. The blue shaded region represents all feasible pairs $(k,\lambda)$ corresponding to a potential $s$-clique in a graph $\Gamma$.

If there exists a point $(k_s^*, \lambda_s^*)$ in the red shaded region such that $\mathrm{CE}_\Gamma(s, k_s^*, \lambda_s^*) > 0$, then by  \Cref{props}, it follows that $\mathrm{CE}_\Gamma(s, k, \lambda) > 0$ for all $(k,\lambda)$ in the black shaded region.

\begin{figure}[!ht]
\begin{minipage}{0.48\textwidth}
\begin{tikzpicture}[>=stealth, thick, scale=0.8]

\fill[blue!30] (1,0.2) rectangle (4.5,3);

\draw[->] (-2,0) -- (4.5,0) node[below right] {$k$};
\draw[->] (-2,0) -- (-2,4.5) node[left] {$\lambda$};

\fill (-2,0) circle (2pt) node[below left] {$0$};

\foreach \x in {-1,0,1,2,3,4}
  \draw (\x,0.1) -- (\x,-0.1) node[below] {};

\foreach \y in {1,2,3,4}
  \draw (-1.9,\y) -- (-2.1,\y) node[left] {};

\node[blue!70!black] at (2.7,1.8) {\footnotesize Feasible region for};
\node[blue!70!black] at (2.7,1.3) {\footnotesize a potential $s$-clique};
\draw[white] (-1, -0.4) -- (-1,0) node[below,text =red!70!black] {$s-1$};
\draw[red!70!black] (-1, -0.1) -- (-1,0.1);
\end{tikzpicture}
\end{minipage}%
\begin{minipage}{0.48\textwidth}
\begin{tikzpicture}[>=stealth, thick, scale=0.8]

\fill[blue!30] (1,0.2) rectangle (4.5,3);

\draw[->] (-2,0) -- (4.5,0) node[below right] {$k$};
\draw[->] (-2,0) -- (-2,4.5) node[left] {$\lambda$};

\fill (-2,0) circle (2pt) node[below left] {$0$};

\foreach \x in {-1,0,1,2,3,4}
  \draw (\x,0.1) -- (\x,-0.1) node[below] {};

\foreach \y in {1,2,3,4}
  \draw (-1.9,\y) -- (-2.1,\y) node[left] {};

\node[blue!70!black] at (2.7,1.8) {};
\fill[red!30] (1,3) rectangle (-1,4.5);
\fill[fill=red!40, draw] (1,3) circle (2pt) node[below left] {};
\draw[white] (-1, -0.4) -- (-1,0) node[below,text =red!70!black] {$s-1$};
\draw[red!70!black] (-1, -0.1) -- (-1,0.1);
\fill[pattern=north east lines, pattern color=black] (0.5,3.5) rectangle (4.5,0.2);
\fill[fill=black, draw] (0.5,3.5) circle (1pt) node[above left] {\footnotesize ($k^*_s, \lambda^*_s$)};
\end{tikzpicture}

\end{minipage}%

\caption{Visual representation of \Cref{visual}.}
\label{feasible region}
\end{figure}
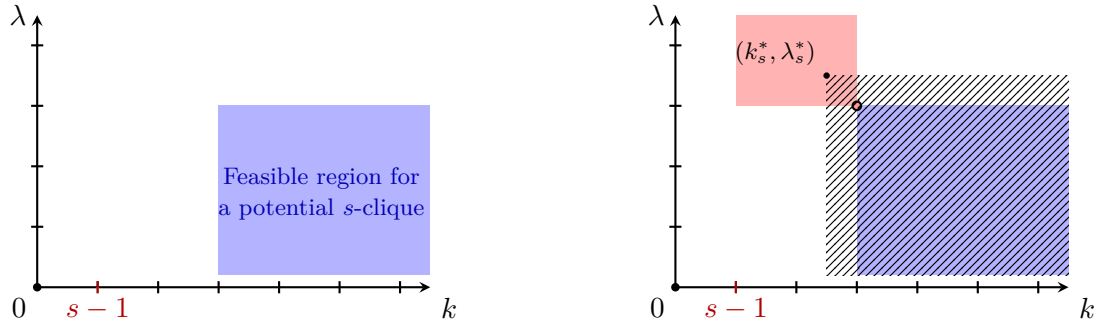

\section{Bounding feasible regions for $(k,\lambda)$}
In this section we will outline methods for bounding $(k,\lambda)$ in different types of graphs, along with some sporadic examples of how to utilise these methods to produce a bound with the CAP.

\subsection{Methods for bounding $\lambda$ in association scheme graphs} 

Since degree valency of graphs of association scheme is always constant, we have a natural bound $k_s^*$ on variable $k$ for all $s$, however, we are required to explore ways to produce an upper bound $\lambda^*_s$ on the variable $\lambda$ (average edge-degree of a potential $s$-clique). The main tool we use to do this will be Delsarte's inner distribution vector. 

    Consider $S$ such that $|S|=s$, we can calculate its inner distribution vector $a$ in some $\Gamma_X$. We have the property \begin{align} s a_i = |(S \times S) \cap R_i|, \label{innerdid} \end{align} for $R_i \in \mathcal{R}$.

    For a non-trivial relation $R_i$, the right-hand side of (\ref{innerdid}) counts each edge in $S$ that belongs to $R_i$ twice, once for each ordering of the endpoints. Let us denote $\hat{a}_i$ as the number of edges in $S$ that belong to the relation $R_i$. We have the following relationship: \begin{align} \frac{s}{2} a_i = \hat{a}_i. \label{edgeconfig} \end{align} We can rescale a given inner distribution vector using this relationship. This allows us to draw some information regarding the configuration of the edges the inner distribution vector represents. We will denote this scaled vector $\hat{a}$. Now also consider an \textit{edge-degree vector} $\mathcal{D}$ for $\Gamma_X$,  such that $\mathcal{D} := (0, \lambda_1, \dots, \lambda_n)$ where $\lambda_i$ is the edge-degree of an edge from $R_i$.\footnote{The edge-degree for an $i$-coloured edge in $\Gamma_X$ can be calculated using the intersection parameters of the scheme}

    The edge-degree vector, along with the scaled inner distribution vector, allows us to observe the following:
	\begin{lemma}
		\label{eds}
		For a given scaled inner distribution vector $\hat{a}$ of $S$, and edge-degree vector $\mathcal{D}$ of $ \Gamma_X$, the value $\hat{a}^\top \mathcal{D}$ sums the total edge-degrees of all the edges of $\Gamma_X$ in the subgraph induced by $S$.
	\end{lemma}

	\Cref{eds} provides a direct method to compute the average edge-degree of an $s$-clique induced by $S$ from its inner distribution vector and intersection numbers of the scheme. Recall, an $s$-clique has $\frac{s(s-1)}{2}$ edges, so we must have $$\langle \lambda \rangle_S = \frac{2 \hat{a}^\top \mathcal{D}}{s(s-1)}.$$

Using the framework from Delsarte's linear program, a new integer linear program can be created, based on the feasibility of the scaled inner distribution vector of a potential $s$-clique $S$ to obtain an upper bound $\lambda^*_s$ on its average edge-degree $\lambda$.  
If the integrality constraints are relaxed to \(\hat a_i\ge 0\), one obtains a linear programming relaxation.
		\begin{align*}
		\textsc{maximise: } & \frac{2 \hat{a}^\top \mathcal{D}}{s(s-1)}, \tag{NEWLP}\\
		\textsc{subject to: } & (\hat{a} Q)_j \ge 0 \textnormal{ for all entries } j, \\
		& \hat{a}_0 = s/2, \\
		& \hat{a}_i = 0 \textnormal{ for } R_i \notin X, \\
		& \sum_{i=1}^n \hat{a}_i = \frac{s(s-1)}{2}, \\
		& \hat{a}_i \in \mathbb{N} \textnormal{ for all } i \ge 1.
	\end{align*}

\begin{rexample}[A graph from the Hamming scheme]
Using the AssociationScheme package in GAP \cite{AssociationSchemes}, we can produce the $H(7,3)$ Hamming scheme. Taking $X = \{R_1,R_2,R_5,R_6,R_7\}$, the graph $\Gamma_X$ is a $(2187,1346)$ regular graph with $\mathcal{D} = (0,1045, 995, 0,0, 815, 775, 715)$. Also from GAP, we can obtain the second eigenmatrix of the scheme. \[ Q=\scriptsize{
\begin{bmatrix}
1 & 14 & 84 & 280 & 560 & 672 & 448 & 128 \\
1 & 11 & 48 & 100 & 80 & -48 & -128 & -64 \\
1 & 8 & 21 & 10 & -40 & -48 & 16 & 32 \\
1 & 5 & 3 & -17 & -16 & 24 & 16 & -16 \\
1 & 2 & -6 & -8 & 17 & 6 & -20 & 8 \\
1 & -1 & -6 & 10 & 5 & -21 & 16 & -4 \\
1 & -4 & 3 & 10 & -25 & 24 & -11 & 2 \\
1 & -7 & 21 & -35 & 35 & -21 & 7 & -1
\end{bmatrix}.}\]
  The Delsarte bound for this graph is $\omega(\Gamma_X) \le 81$, so considering a potential $81$-clique we can evaluate our new linear program to get an upper bound on the average edge-degree of $\lambda^*_s = 833.75$. Now evaluating $\CE_{\Gamma_X}(81, 1346, 833.75) \approx 2.7 \times 10^6 > 0$. By \Cref{visual} a clique of size $81$ cannot exist.
\end{rexample}

Aside from producing an upper bound on $\lambda$, note that what was the objective of Delsarte's program has now become a constraint. So, similar to Delsarte's LP bound, this linear program can produce a bound on $\omega(\Gamma_X)$, however in a slightly different way. Evidently we have:

\begin{lemma}
    For association scheme graph $\Gamma_X$, if the linear program \textnormal{NEWLP} is infeasible for a certain value of $s$, then $\omega(\Gamma) \le s-1$.
\end{lemma}

\subsection{Bounding $\lambda$}
Here we provide some general methods to produce bounds for the variable $\lambda$ in graphs that do not arise from association schemes. For a graph $\Gamma$ we can define the edge-degree sequence of all its edges $\mathcal{D}_\Gamma := \lambda_1, \lambda_2, \dots \lambda_m$ where for $i<j$, $\lambda_i \le \lambda_j$ and $m = |E(\Gamma)|$. 
 	Since there is no underlying structure on the edges, we must index each edge individually. We call the interval $\left[\lambda_1, \lambda_m\right]$ the \emph{trivial bounds} for $\lambda$. These can be refined by noting that an $s$-clique contains $\binom{s}{2}$ edges. Therefore,
	
	\[
  \lambda \in 
  \left[
    \frac{\sum_{i=1}^{\binom{s}{2}} \lambda_i}{\binom{s}{2}}, \;
    \frac{\sum_{i=m-\binom{s}{2}+1}^{m} \lambda_i}{\binom{s}{2}}
  \right].
\]
	These bounds are straightforward to obtain, and in some cases they can be sharpened. For a given vertex $u \in V(\Gamma)$ of degree at least $s-1$, we define an \textit{$s$-restricted edge-degree sequence} $\mathcal{D}_s(u)$. This is the increasing sequence $\lambda_{1_u}, \dots, \lambda_{(s-1)_u}$ of the $(s-1)$ largest edge-degrees among edges incident with $u$. Such a sequence can be computed for every $u \in V(\Gamma)$.

	From these, we construct a descending \emph{nested $s$-restricted edge-degree sequence} using all the vertices of $\Gamma$: \[
  \mathcal{D}_s(u_1), \mathcal{D}_s(u_2), \dots, \mathcal{D}_s(u_v),
\]
where the vertices are ordered so that for $a < b$,
\[
  \sum \mathcal{D}_s(u_a) \ge \sum \mathcal{D}_s(u_b).
\]

	\begin{theorem}
	For a graph $\Gamma$, clique of size $s$ and nested $s$-restricted edge-degree sequence \[\mathcal{D}_s(u_1), \mathcal{D}_s(u_2), \dots, \mathcal{D}_s(u_v)\] we can produce an upper bound $\lambda^*_s$ on $\lambda$ such that \[\lambda_s^* = \frac{\sum_{i=1}^s \sum \mathcal{D}_s(u_i)}{s(s-1)}.\]
	\end{theorem}
    
	\begin{proof}
	Consider the subsequence of $\mathcal{D}_s(\Gamma)$ of the form $\mathcal{D}_s(u_{1}), \dots, \mathcal{D}_s(u_s)$.  
The set of vertices $\{u_{1}, \dots, u_s\}$ correspond to the subset of $s$ vertices whose $(s-1)$ largest edge-degrees together give the maximum possible total among all $s$-subsets of vertices. Thus, for any $s$-clique $C$, we have that 
    \[  
    \sum_{u\in C}\sum \mathcal{D}_s(u)  
    \le  
    \sum_{i=1}^s\sum \mathcal{D}_s(u_i).  
    \]
    
    This means we have the bound
\[
  \lambda \le 
  \frac{\sum_{i=1}^{s} \sum \mathcal{D}_s(u_i)}{s(s-1)},
\]
and take the right-hand side to be $\lambda^*_s$.
	\end{proof}
    
	Sharpness of bounds for regular graphs often relies on bounding $\lambda$ as low as possible for each $s$. This task can be computationally heavy, however it is faster than searching for the largest clique as vertex number increases. The methods above are particularly effective at producing sporadic examples where the generalised CAP improves on the Hoffman bound, or for single-graph analysis.

\subsection{Bounding $k$}

Similar to the edge-degree sequence, we can consider a \textit{degree sequence} to bound $k$.	For a graph $\Gamma$, the degree sequence $\mathcal{K}_\Gamma := k_1, k_2, \dots k_v$ where $k_i = |\Gamma(u_i)|$ for $u_i \in V(\Gamma)$. The sequence is given in ascending order such that for $i<j$, $k_i \le k_j$. For a potential $s$-clique, only the average of the largest $s$ elements in $\mathcal{K}_\Gamma$ needs to be considered: $$k \in \left[\frac{\sum_{i=1}^s k_i}{s}, \frac{\sum_{i=v-s+1}^v k_i}{s}\right].$$ In this case $\left[\delta_\Gamma, \Delta_\Gamma\right]$ is the trivial bounds for $k$. Consistent with our earlier findings, these processes are useful at producing sporadic results for when the CAP bound outperforms the Haemers bound.

\begin{rexample}[A random regular graph]
We can use methods discussed to investigate the clique number of a random regular graph. Take $\Gamma$ to be a $(100,25)$ regular graph. 
We generated this graph using the NetworkX random regular graph function with seed 600. The Hoffman bound for this graph is $11$. We have a natural bound $k^*_s = 25$ and Hoffman bound $\lambda^*_s=11$ for all $s$. The clique existence polynomial gives $\CE_\Gamma(12, 25, 11) = 9856 > 0$. By \Cref{visual}, this CAP bound is the same as the Hoffman bound, however more work can be done to produce a tighter bound on $\lambda$.
\\
\\
Assuming a potential clique of size $9$, we can calculate $\sum_{i=1}^{9}\sum \mathcal{D}_9(u_i) = 624$ which gives a bound of $\lambda_9^* = 8.\bar6$. The clique existence polynomial gives $\CE_\Gamma(9, 25,8.\bar6) = 1456 > 0$. The CAP bound now gives $ \omega(\Gamma) \le 8$. This bound is reasonable and in fact $\omega(\Gamma)= 6$. Looking at \Cref{finalbound} we can see this information graphically. 
\end{rexample}

\begin{figure}[!ht]

\centering

\begin{tikzpicture}

\begin{axis}[
    xlabel={$s$},
    ylabel={$\mathrm{CE}_\Gamma(s)$},
    xtick={8,9,10,11,12,13},
    domain=7:13,
    samples at={8,9,10,11,12},
    width=10cm,
    height=6cm,
    yticklabel style={/pgf/number format/fixed},
    extra y ticks={0},
    extra y tick labels={0},
    extra y tick style={grid=major},
    axis lines=middle,
    grid=both,
    set layers,
    layers/axis on top/.style={axis on top},
    unbounded coords=jump,
    thick,
    legend style={
        at={(0.98,0.98)},
        anchor=north east
    }
]

\addplot+[
    blue,
    only marks,
    mark=*,
    mark size=2.5pt
]
{ ((100 - 53*x + 2*x^2)^2)
  - 4*(100 - x)*(-13*x + 14*x^2 - x^3)
  - ((100 - x)*(1 - abs(-1 + 2*((-100 + 53*x - 2*x^2)/(2*(100 - x))
     - floor((-100 + 53*x - 2*x^2)/(2*(100 - x)))))))^2
};

\addlegendentry{$\lambda=13$}

\addplot+[
    red,
    only marks,
    mark=*,
    mark size=2.5pt
]
{ ((100 - 53*x + 2*x^2)^2)
  - 4*(100 - x)*(-(10.6666666667*x) + 11.6666666667*x^2 - x^3)
  - ((100 - x)*(1 - abs(-1 + 2*((-100 + 53*x - 2*x^2)/(2*(100 - x))
     - floor((-100 + 53*x - 2*x^2)/(2*(100 - x)))))))^2
};

\addlegendentry{$\lambda=8.\bar6$}

\end{axis}

\end{tikzpicture}

\caption{$\mathrm{CE}_\Gamma(s, 25, \lambda)$ for $\lambda = 13, 8.\bar6$}
\label{finalbound}

\end{figure}
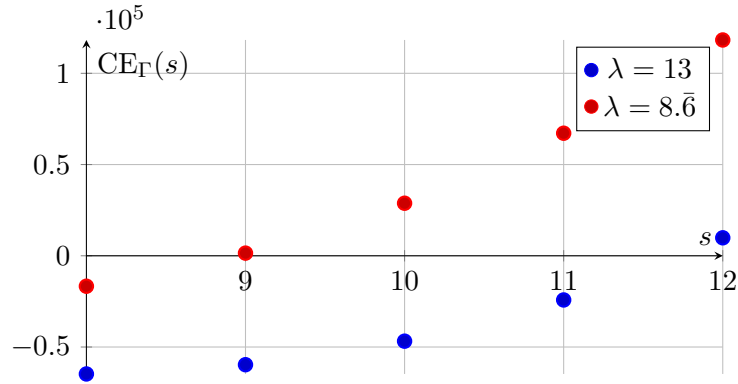

\section{Infinite families}


\subsection{Edge-regular}

 First, consider some infinite family of $(v,k,\lambda)$ edge-regular graphs $\mathcal{U}$, such that there is an upper bound $N$ on $\lambda$ for all $\Gamma \in \mathcal{U}$. In this instance, we see that
        \[
        \CA_\Gamma(0,N+3)=(N+3)(N+2)(\lambda-N-1)\le -(N+3)(N+2)< 0,
        \]
for all $\Gamma \in \mathcal U$. We get the following lemma.

\begin{lemma}
\label{trivial}
For an infinite family of $(v,k,\lambda)$ edge-regular graphs $\mathcal{U}$, if there exists an upper bound $N$ on $\lambda$ for all $\Gamma \in \mathcal{U}$ then the CAP proves that $\omega(\Gamma) \le N+2$ for all $\Gamma \in \mathcal{U}$.  
\end{lemma}

While this bound is trivial (and almost never sharp for the CAP), for some families of graphs, spectral bounds cannot always reproduce these trivialities. We will consider two infinite families of edge-regular graphs with this property in this section.

For odd $q$, consider the group action of $\textnormal{PGL}(2, q)$, as a subgroup of $\textnormal{PGL}(3, q)$, on the secant lines of a non-degenerate conic  $\mathcal{O}$ in $\textnormal{PG}(2,q)$. This action is generously transitive. We define $\textnormal{PGL}(2,q)^+$ as the Schurian association scheme generated by this action, where relations are given by its orbital graphs. One of the orbital graphs in particular we focus on, we refer to as $\Gamma_{q^+}$, where $\{a,b\} \in E(\Gamma_{q^+})$ if and only if $a$ and $b$ are polar conjugates with respect to the polarity induced by $\mathcal{O}$. Since $\Gamma_{q^+}$ is a graph of a single relation in an association scheme, it must be edge-regular.

Evidently, we can also see that $\Gamma_{q^+}$ is isomorphic to an induced subgraph $\mathcal{H}$ of the Erd\H{o}s–R\'enyi polarity graph\footnote{known to have minimum eigenvalue of $-\sqrt{q}$ \cite{erpg}} ($\textnormal{ER}_q$), where $V(\mathcal{H})$ is the external points of $\mathcal{O}$, corresponding to the secant lines under the polarity mapping with respect to $\mathcal{O}$.

Since $\Gamma_{q^+}$ is isomorphic to $\mathcal{H}$, by interlacing of eigenvalues we can retrieve the following:
	\begin{corollary}
		\label{minev}
		The minimum eigenvalue of $\Gamma_{q^+}$ must be greater than or equal to the minimum eigenvalue of $\textnormal{ER}_q$, $-\sqrt{q}$.
	\end{corollary}

For a visual understanding of \Cref{minev}, we generated the eigenvalues of $\Gamma_{q^+}$ for the odd prime powers from $7$ to $89$ using GAP. It is indeed true that these values are always greater than or equal to $-\sqrt{q}$, however in fact, they seem generally very close to $-\sqrt{q}$. See the plot in \Cref{plus} along with the lower bound of $-\sqrt{q}$. 

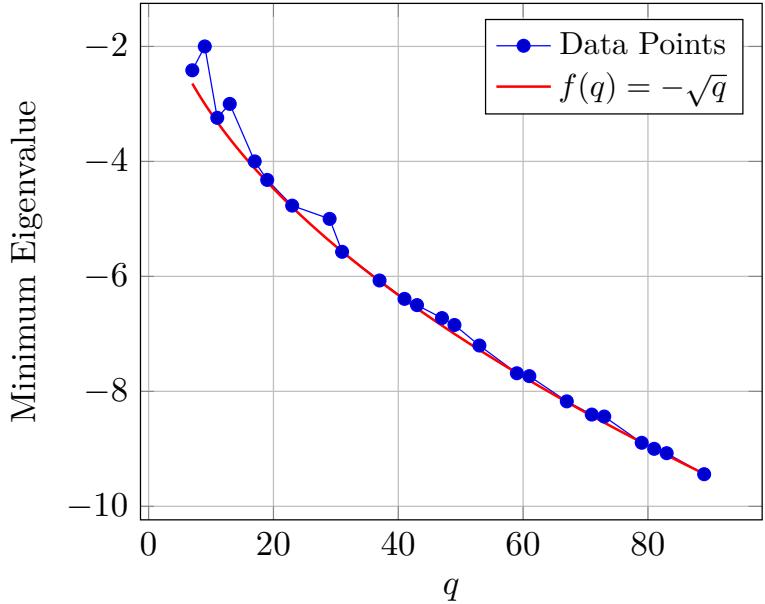
\begin{figure}[!ht]
	\captionsetup{width=.6\linewidth}
	\begin{center}
	\begin{tikzpicture}[scale=1.2]
	
    \begin{axis}[
        xlabel={$q$},
        ylabel={Minimum Eigenvalue},
        grid=major,
        legend pos=north east,
        legend style={font=\small}
    ]
    \addplot coordinates {
		(7.0, -2.4150390625)
 		(9.0, -2.0)
        (11.0, -3.2421875)
        (13.0, -3.0)
        (17.0, -4.0)
        (19.0, -4.322265625)
        (23.0, -4.76953125)
        (29.0, -5.0)
        (31.0, -5.5732421875)
        (37.0, -6.07177734375)
        (41.0, -6.3916015625)
        (43.0, -6.5)
        (47.0, -6.7255859375)
		(49, -224317/32768)
        (53.0, -7.203125)
        (59.0, -7.685546875)
        (61.0, -7.735595703125)
        (67.0, -8.17431640625)
        (71.0, -8.404296875)
        (73.0, -8.438232421875)
        (79.0, -8.89501953125)
		(81,-9)
        (83.0, -9.076171875)
        (89.0, -9.44140625)
    };
    \addlegendentry{Data Points}

    \addplot[domain=7:90, samples=90, smooth, red, thick] { -sqrt(x) };
    \addlegendentry{$f(q) = -\sqrt{q}$}

    \end{axis}
\end{tikzpicture}

\end{center}
\caption{Plot displaying the minimum eigenvalue of $\Gamma_{q^+}$ for the odd prime powers from $7$ to $89$.}
	\label{plus}
\end{figure}

First we provide the following lemma:
\begin{lemma}
		\label{edge}
	A pair of vertices in $\textnormal{ER}_q$ cannot have two or more common neighbours. In particular, this implies for any value of $s$ and $q$, we have an upper bound of $1$ on the edge-degree of any edge.
	\end{lemma}
	\begin{proof}
	Recall, an edge in $\textnormal{ER}_q$ is between two vertices if they are polar conjugates. This means that a pair of vertices have a common neighbour in $\textnormal{ER}_q$ if they have a common polar conjugate point. Consider vertices $p$, $p'$. Assume they have two common polar conjugate points $a$, $b$. Then $a$ and $b$ will lie not only on $p^\perp$ but also on $p'^\perp$. Since two points define a unique line, we have that $p^\perp = p'^\perp$. This is a contradiction since $\perp$ is a bijection. This means such an $a$ and $b$ cannot exist, and therefore a pair of vertices in $\textnormal{ER}_q$ cannot have two or more common neighbours.
	\end{proof}

This provides a bound of $\lambda = 1$ for all $\Gamma_{q^+}$.
These graphs were analysed by Parsons initially \cite{parse}, and in fact we can represent $(v,k)$ of $\Gamma_{q^+}$, all in terms of $q$. Specifically, $v = \frac{q(q+1)}{2}$ and $k = \frac{q-1}{2}$.

\begin{theorem}
		\label{edge_reg_1}
	The CAP bound improves on the Delsarte bound for $\Gamma_{q^+}$ for infinitely many $q$.
\end{theorem}

\begin{proof}
It follows directly from \Cref{trivial} and \Cref{edge} that the CAP proves $\omega(\Gamma_{q^+})\le 3$.  
     Since the Delsarte bound of $\Gamma_{q^+}$ grows like $\sqrt q$, it exceeds $3$ for infinitely many $q$.
     To see why, the Delsarte bound of $\Gamma_{q^+}$ can be written $\omega(\Gamma_{q^+})\le 1-\frac{k}{\gamma_v}$
     where $k=(q-1)/2$. Since $\gamma_v\ge -\sqrt{q}$, we have $|\gamma_v|\le \sqrt{q}$ and hence the Delsarte
     bound is at least \[1+\frac{q-1}{2\sqrt{q}}\] which tends to infinity.
	\end{proof}

For odd $q$, a similar result can be replicated when considering, the group action of $\textnormal{PGL}(2, q)$, as a subgroup of $\textnormal{PGL}(3, q)$, on the external lines of a non-degenerate conic  $\mathcal{O}$ in $\textnormal{PG}(2,q)$.

Again, one orbital graph, which we denote as $\Gamma_{q^-}$, contains an edge between two external lines if and only if they are polar conjugate. Similarly to $\Gamma_{q^+}$, we have that $\Gamma_{q^-}$ is isomorphic to an induced subgraph $\mathcal{N}$ of $\textnormal{ER}_q$, where $V(\mathcal{N})$ is the set of non-external, non-absolute points of $\mathcal{O}$, corresponding to the external lines under the polarity mapping with respect to $\mathcal{O}$. For $\Gamma_{q^-}$ we have $v = \frac{q(q-1)}{2}$ and $k= \frac{q+1}{2}$.

\begin{theorem}
	\label{edge_reg_2}
		The CAP bound improves on the Delsarte bound for $\Gamma_{q^-}$ for infinitely many $q$.
\end{theorem}


\begin{proof}
	 See proof of \Cref{edge_reg_1}.
\end{proof}

\begin{figure}[!ht]
	\captionsetup{width=.6\linewidth}
\begin{center}
\begin{tikzpicture}[scale=1.2]
	
    \begin{axis}[
        xlabel={$q$},
        ylabel={Minimum Eigenvalue},
        grid=major,
        legend pos=north east,
        legend style={font=\small}
    ]
    \addplot coordinates {
       (7.0, -2.0)
 (9.0, -3.0) (11, -3) (13, -3.5625) (17, -3.75) (19, -3.953125) (23, -4.78125)  (25.0, -5.0)
 (27.0, -4.6484375) (29, -5.171875) (31, -5.4453125) (37, -5.984375) (41, -6.32421875) (43, -6.4609375) (47, -6.75) (49, -7) (53, -7.234375) (59, -7.640625) (61, -7.796875) (67, -8) (71, -8.3671875) (73, -8.5234375) (79, -8.75390625) (81.0, -9.0) (83, -9.064453125) (89, -9.3515625)
    };
    \addlegendentry{Data Points}
    
    \addplot[domain=7:90, samples=100, smooth, red, thick] { -sqrt(x) };
    \addlegendentry{$f(q) = -\sqrt{q}$}
    
    \end{axis}
\end{tikzpicture}
\end{center}
\caption{Plot displaying the minimum eigenvalue of $\Gamma_{q^-}$ for the odd prime powers from $7$ to $89$.}
\label{mini}
\end{figure}
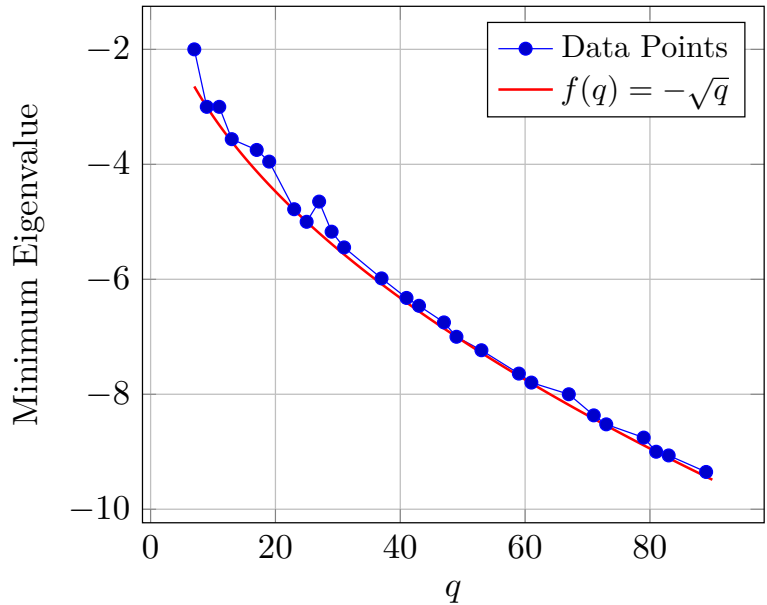

\subsection{Regular graphs}

 \begin{corollary}
	\label{method}
	For a connected $k$-regular graph $\Gamma$, consider the $v$ vertex graph $\overline{\Gamma \square C_n}$ for large enough $n$. We have that its minimum eigenvalue is given by: $$\gamma'_{v} = -1-2\cos\left(\frac{2\pi}{n}\right)-k.$$
	\end{corollary}
    
	\begin{proof}
		Since $\overline{\Gamma \square C_n}$ is regular (with valency $k+2$) we can use its second largest eigenvalue to determine $\gamma'_{v}$. To find its second largest eigenvalue we will consider $\Gamma$ and $C_n$ separately. Let us use $r$ to denote its second largest eigenvalue. For $C_n$ we have $2,2\cos\left(\frac{2\pi}{n}\right)$ as its two largest eigenvalues. This means for $\Gamma \square C_n$, we have $k+2$ as our largest eigenvalue, and either $k+2\cos\left(\frac{2\pi}{n}\right)$ or $r+2$ as its second largest eigenvalue. It is clear that as $n \rightarrow \infty$, that $2\cos\left(\frac{2\pi}{n}\right) \rightarrow 2$. This means for large enough $n$, since $r \neq k$, the second largest eigenvalue of $\Gamma \square C_n$ is given as $k+2 \cos \left(\frac{2\pi}{n}\right)$. Let $n$ be large enough such that this is the case. We have \[\gamma'_{v} = -1-2\cos\left(\frac{2\pi}{n}\right)-k. \qedhere\]
	\end{proof}

    Now observe what happens as $n$ approaches infinity.

	\begin{corollary}
	\label{CAPHoff}
	For $\Gamma \square C_n$, as $n \rightarrow \infty$, the Hoffman bound approaches $3+k$.
	\end{corollary}
	\begin{proof}
	From \Cref{method}, we can see that as $n \rightarrow \infty$, we have that $\gamma'_{v} \rightarrow -3-k$. Let $|V(\Gamma)| = u$, which means $|V(\Gamma \square C_n)|= v = un$ . Using the Hoffman bound we obtain the following: \begin{align*} \frac{-v \gamma'_{v}}{v-(k+2)-1-\gamma'_{v}}  = \frac{-un \gamma'_{v}}{un-k-3-\gamma'_{v}} = \frac{- \gamma'_{v}}{1-\frac{k+3+\gamma'_{v}}{un}}.\end{align*} We can see as $n \rightarrow \infty$ the Hoffman bound approaches $-\gamma'_{v}$. This means as $n \rightarrow \infty$, the Hoffman bound for $\Gamma \square C_n$ approaches $3+k$.
	\end{proof}

    We can see that the Hoffman bound performs poorly for graphs of this nature, and produces a redundant solution. The largest possible size of a clique of $\Gamma \square C_n$ is at most $k+1$, for any $\Gamma$. We can show that the generalised CAP is much more robust when bounding these types of graphs.
This leads to: 

\begin{theorem}
		\label{reg_1}
	The generalised CAP bound improves on the Hoffman bound for $\Gamma \square C_n$, when $\Gamma$ is a $k$-regular graph, for large enough $n$.
\end{theorem}

	\begin{proof}
	From Corollary \ref{CAPHoff}, there exists some value $n$ for which the Hoffman bound is greater than $k+2$. Let us consider $\Gamma \square C_n$ for this value of $n$. We can input some parameters into the generalised CAP of the form $\CA_{\Gamma \square C_n}(x,s,\lambda)$. Since $\Gamma \square C_n$ is $k+2$ regular, we have omitted the average degree variable from the polynomial. We set $s$ as $k+2$, a tighter bound than the Hoffman bound for infinitely many values of $n$. 
    Every edge of $\Gamma\square C_n$ has at most $k-1$ common neighbours. Indeed, if the edge lies inside a $\Gamma$-fibre, then its common neighbours are common neighbours of an edge of $\Gamma$, and there are at most $k-1$ of these. If the edge lies in the $C_n$-fibre, then it has no common neighbours for \(n\ge 4\). Hence $\lambda^\ast:=k-1$ is an upper bound
    on the potential edge-degree for a clique in $\Gamma \square C_n$.  This gives us a polynomial in the form: \begin{align*} \CA_{\Gamma \square C_n}(x) = -(1 + k) (2 + k) - 2 (2 + k) x + (-2 - k + v) x (1 + x). \end{align*} where $v$, as usual is the size of the vertex set of $\Gamma \square C_n$. We can see that for $x=0$, the generalised CAP produces a negative value: \begin{align*} \CA_{\Gamma \square C_n}(0) = -(1 + k) (2 + k). \end{align*} Therefore, a clique of size $k+2$ cannot exist. The generalised CAP bound improves on the Hoffman bound in this case.
	\end{proof}

\subsection{Non-regular}

Assuming that a graph $\Gamma$ contains a universal vertex, we can evaluate the Haemers bound for $\Gamma$. We get the following: 
\begin{align*} 
		\omega(\Gamma) \le \frac{-v \gamma'_1\gamma'_{v}}{(v-(v-1)-1)^2-\gamma'_1\gamma'_{v}} = v.
	\end{align*}
	The Haemers bound does not deal well with these types of graphs. The generalised CAP is more robust. 
 

    \begin{theorem}
	\label{nonreg_1}
	For a non-trivial graph $\Gamma$ with a single universal vertex, the generalised CAP bound improves on the Haemers bound.
	\end{theorem}

	\begin{proof}
	First we claim that the maximum edge-degree of a graph of this type is $v-3$. If an edge had degree $v-2$, then both vertices of the edge would be universal vertices, which is a contradiction. Substituting in the Haemers bound $v$ as $s$, and $v-3$ as $\lambda$ means the generalised CAP has the form $$\CA_\Gamma(x,v,k,v-3) = -(v(v-1))-2(1+k-v)vx. $$ Setting $x=0$ means that we do not need to consider potential values of the $k$ variable. We get \begin{align*} \CA_{\Gamma}(0)=  -v(v-1). \end{align*} For any non-trivial graph, this value is always negative, therefore the generalised CAP improves on the Haemers bound for graphs of this type.
	\end{proof}

\bibliographystyle{abbrv}
\bibliography{refs}

\end{document}